\documentclass[12pt]{amsart}
\usepackage{graphicx} 
\usepackage{amsmath}
\usepackage{amssymb}
\usepackage{amsfonts}
\usepackage{mathtools}
\usepackage{amsthm}  
\usepackage{dsfont}
\usepackage[colorlinks]{hyperref}
\usepackage{mathrsfs}
\usepackage{anysize}
\marginsize{1.5cm}{1.5cm}{2cm}{2cm}

\newcommand{\R}{\mathbb{R}}

\newcommand{\Z}{\mathbb{Z}}

\newcommand{\diam}[1]{\operatorname{diam}\left(#1\right) }
\newcommand{\supp}[1]{\operatorname{supp}\left(#1\right) }
\newcommand{\dist}[1]{\operatorname{dist}\left(#1\right) }

\newcommand{\one}{\mathds{1}}
\DeclareMathOperator*{\esssup}{ess\,sup}
\newcommand{\Ss}{\mathcal{S}}
\newcommand{\AS}{\mathcal{A}_{\Ss}}

\newcommand{\D}{\mathscr{D}}

\theoremstyle{plain}
\newtheorem{theorem}{Theorem}[section] 

\newtheorem{problem}{Problem}[section]
\newtheorem*{theorem*}{Theorem}

\newtheorem{lemma}[theorem]{Lemma}

\theoremstyle{definition} 
\newtheorem{remark}[theorem]{Remark} 

\numberwithin{equation}{section}

\newtheorem{ltheorem}{Theorem}

\title[Sparse via Calder\'on-Zygmund: Dini kernels]{Sparse domination via the Calder\'on-Zygmund decomposition: the example of Dini-smooth kernels}

\author[F. Ballesta-Yag\"ue]{Fernando Ballesta-Yag\"ue}
\address[F. Ballesta-Yag\"ue]{Departamento de An\'alisis Matem\'atico y Matem\'atica Aplicada, Facultad de Ciencias Matem\'aticas, Universidad Complutense de Madrid \hfill\break\indent Pl. de las Ciencias 3, 28040 Madrid, Spain}
\email{ferballe@ucm.es}

\author[J.M. Conde Alonso]{Jos\'e M. Conde-Alonso}
\address[J.M. Conde Alonso]{Departamento de Matem\'aticas, Universidad Aut\'onoma de Madrid \hfill\break\indent 
 C/ Francisco Tom\'as y Valiente, 28049 Madrid, Spain}
\email{jose.conde@uam.es}

\allowdisplaybreaks
\begin{document}

\begin{abstract}
    In this expository article, we briefly survey the main known schemes of proof of sparse domination principles within harmonic analysis. We then use the one based on the Calder\'on-Zygmund decomposition to prove a dual sparse domination estimate for Calder\'on-Zymgund operators with Dini-smooth kernels, with an eye on the difficulties that arise when trying to transfer the argument to spaces with nondoubling measures. 
\end{abstract}

\maketitle

\section*{Introduction}

The last decade has seen the rapid development of the sparse domination technique within harmonic analysis. Given a measure space $(\Omega,\Sigma,\mu)$ and $0<\eta<1$, a family $\Ss\subset \Sigma$ is called $\eta$-sparse if for each $Q \in \Ss$, there exists $E_Q \subset Q$ with two properties:
\begin{itemize}
    \item $\mu(E_Q) \geq \eta \mu(Q)$ for each $Q\in\Ss$.
    \item For each $Q\neq R \in \Ss$, $E_Q$ and $E_R$ are disjoint.
\end{itemize}
In the simplest version, a sparse operator is an averaging one of the form 
$$
\AS f(x) = \sum_{Q\in \Ss} \langle f \rangle_Q \one_Q(x), \quad \mbox{where} \quad \langle f \rangle_Q = \frac{1}{\mu(Q)}\int_Q f d\mu,
$$
and the family $\Ss$ is $\eta$-sparse. Whenever $\Ss$ is composed of reasonable sets, $\AS$ is a self-adjoint operator which is bounded on $L^p(\mu)$ for $1<p<\infty$ and which is of weak-type $(1,1)$. A sparse domination for an operator $T$ refers to an inequality of one of the kinds below:
$$
\textbf{(i)} \;\|Tf\|_{\mathbb{X}} \lesssim \|\AS f\|_{\mathbb{X}}, \quad \textbf{(ii)} \;\left|\langle Tf,g\rangle\right| \lesssim \langle \AS |f|,|g|\rangle, \quad \textbf{(iii)} \;|Tf(x)| \lesssim \AS |f|(x) \; a.e.
$$
$\mathbb{X}$ is some appropriate Banach space of functions in inequality \textbf{(i)}, which is weaker than \textbf{(ii)}, while that is clearly implied by \textbf{(iii)}. In all the above, the family $\Ss$ is $\eta$-sparse for some fixed valued of $\eta$, but the family itself does depend on $f$ and $g$. A sparse domination --of one of the kinds \textbf{(ii)} and \textbf{(iii)}, which are the ones of more interest to us here-- can be viewed as a somewhat strong quantification of the weak-type behavior of the operator $T$, if the geometry of the sets in $\Ss$ is nice enough. In this paper, we will only consider families of cubes with sides parallel to the coordinate axes.

Sparse domination has been first studied in $\R^d$, with its importance coming from weighted inequalities. This is so because $\AS$ --when $\Ss$ is a family of cubes-- can be shown to easily satisfy the $A_2$ conjecture: if $w$ belongs to the Muckenhaupt class $A_2$, then a beautiful argument from \cite{CMP2012} implies that
$$
\|\AS f\|_{L^2(w)} \lesssim [w]_{A_2} \|f\|_{L^2(w)},
$$
with implicit constant independent of $w$. This matches the sharp weighted bound for Calder\'on-Zygmund singular integral operators, proved in full generality in \cite{Hy2012}. Therefore, a sparse domination for a Calder\'on-Zygmund operator $T$ of any of the kinds \textbf{(i)-(iii)} above yields a new proof of the $A_2$ theorem, which was a major motivation in the beginning of the theory. Later, sparse domination results have been used to prove sharp weighted inequalities beyond the $A_2$ conjecture, including weighted weak-type estimates or estimates in non-Euclidean contexts; some weighted estimates can even be attained without the use of extrapolation, using only sparse domination techniques \cite{Moen}. Even --previously unknown-- unweighted weak-type estimates have been shown as a consequence of delicate sparse domination arguments, proving the applicability of the theory beyond its original goals. A very nice account of the development of the field can be found in \cite{Pe2019}. 

In this note, we focus on sparse domination principles for Calder\'on-Zygmund singular integral operators on $\R^d$. A Calder\'on-Zygmund operator is an $L^2$-bounded linear map that formally admits the representation
$$
Tf(x) = \int_{\R^d} K(x,y) f(y) \, dy,
$$
for appropriate functions $f$ and points $x$, and whose kernel $K$ satisfies $|K(x,y)| \lesssim |x-y|^{-d}$ when $x\neq y$ plus an additional smoothness condition, that we will make precise below and which can be the source of meaningful technical difficulties. There are several ways in which one can achieve sparse domination for these operators:

\begin{itemize}
    \item \textbf{Sparse domination based on median oscillation decomposition:} Let $m_Q(f)$ denote a median of a function $f$ over a cube $Q$, and $\omega_\lambda(f;Q)$ its oscillation off a $\lambda$-fraction, that is
    $$
    \omega_\lambda(f;Q)= \inf_{A \subset Q; |A| \geq (1-\lambda)|Q|} \sup_{x,y\in A}|f(x)-f(y)|.
    $$
    Lerner's median oscillation formula states the following:
    $$
    |f(x)-m_Q(f)| \lesssim \sum_{R \in \Ss}\omega_{\lambda(d)}(f;Q),
    $$
    for some sparse family $\Ss$ and some dimensional quantity $0<\lambda(d)<1$. Applying the above to $T$ yields a pointwise bound for it in terms of certain non-cancellative Haar shifts. This allows one to prove sparse domination in the sense \textbf{(i)}, yielding the aforementioned simplification of the $A_2$ theorem \cite{Le2013,Le2013b}. Later, the same technique allowed the first sparse domination proofs in the pointwise sense \textbf{(iii)} \cite{CAR2016,LeNa2019}. In all these results, the kernel smoothness is required to be the so-called logarithmic Dini condition, a somewhat strong one. 
    \item \textbf{Sparse domination based on grand maximal functions:} In \cite{La2017} and the works that followed it, variants of the following (local) grand maximal operator are introduced:
    $$
    \mathcal{M}_{Q}f(x) = \sup_{x\in R \subset Q} \esssup_{y \in R}|T(f\one_{3Q\setminus 3R})(y)|.
    $$
    This grand maximal operator is of weak-type $(1,1)$, and so it can be used to run a stopping time argument that simplifies the proof above and eliminates the passage through Haar shift operators. Using this scheme, Lacey was able to achieve sparse domination for operators with kernels satisfying Dini smoothness. The method was further streamlined in subsequent works \cite{Le2016,LeO2020}.
    \item \textbf{Sparse domination based on the Calder\'on-Zygmund decomposition:} One scheme of proof of sparse domination of the kind \textbf{(ii)} uses infinitely many Calder\'on-Zygmund decompositions and was initially introduced to study non-kernel operators \cite{BFP2016}. Later, it was used to study dyadic variants of Calder\'on-Zygmund operators \cite{CuDPOu2018} and rough singular integrals \cite{CACDPO2017}. This scheme has later evolved using local, single-scale behavior of the operators under consideration as a tool to achieve the estimates \cite{BRS2024}. This class of single-scale gains of the operators is usually termed $L^p$-improving. 
\end{itemize}

There are of course many other recent developments that carry the theory much further, and we refer again to \cite{Pe2019} for a more systematic account of sparse domination for objects of highly oscillatory nature, Banach/operator valued settings, the vector valued theory with matrix weights, and other settings. All in all, a large part of the analysis of interesting operators near the $L^1$ endpoint has been transferred to the new language of sparse domination. One outstanding remaining question of interest to us, however, is that of nonhomogeneous spaces, that we discuss in more detail in Section \ref{sec:sec3}. We want to contextualize some of the difficulties encountered in that setting in our proof below. Indeed, in this note we use the scheme of proof, neatly outlined in \cite{CuDPOu2018}, to study smooth Calder\'on-Zygmund operators. Our main result below is strictly contained in many others cited above, one of them being \cite{CACDPO2017}, where a similar approach is taken. We also follow arguments from \cite{BRS2024} at times.

\begin{ltheorem}\label{theoremA}
Let $T$ be a Calder\'on-Zygmund operator whose kernel satisfies the Dini condition, and $0<\eta<1$ small enough. For each pair $(f_1, f_2)$ of nonnegative, bounded functions supported on some cube $Q_0\in \D$, there exists an $\eta$-sparse family of cubes $\Ss$ such that
\[
\left|\left\langle T f_1, f_2\right\rangle\right| \lesssim \sum_{Q \in \Ss}\left\langle f_1\right\rangle_Q\left\langle f_2\right\rangle_Q|Q|.
\]    
\end{ltheorem}

The proof of Theorem \ref{theoremA} splits the kernel $K$ of the Calder\'on-Zygmund operator $T$ using a radial Littlewood-Paley decomposition. This was already used by \cite{HRT2017} in order to prove weighted inequalities for rough singular integrals, and is the point of view adopted in \cite{CACDPO2017} as well. Then, a careful sequence of Whitney decompositions of the relevant level sets allows one to estimate a piece of the operator in each step using the simultaneous Calder\'on-Zygmund decomposition of the two functions and is used to construct $\Ss$. As has been noted elsewhere \cite{LaMe2017}, the family $\Ss$ in the statement must depend on the pair of functions $(f_1,f_2)$, but does not depend on $T$. 

The scheme of proof described above has been carried over several times in the literature. We have chosen to present the (smooth) Calder\'on-Zygmund case --with pointwise smoothness bounds-- because we find that it illustrates the method in a clear way, without the additional technicalities of other settings, like the rough one of \cite{CACDPO2017}. On the other hand, it does incorporate geometric difficulties, like interactions between cubes that are near one another, which are not present in the dyadic setting of \cite{CuDPOu2018}. Therefore, we believe that it is better to display them because it is often a source of difficulties when one tries to transfer the full scheme to other settings or the analysis of other operators. Our proof is reasonably self-contained, relying only on the classical Whitney decomposition, while all other details are provided. Finally, the presentation of said details allows us to discuss the nonhomogeneous case, where significant challenges remain open and where it is not even clear that the sparse strategy is the right one to deal with weighted inequalities. 

The rest of this note is structured as follows: in Section \ref{sec:sec1} we explain the two main (classical) tools that appear in the proof of the sparse domination result: Whitney coverings and a modification of the Calder\'on-Zygmund decomposition. We also introduce Calder\'on-Zygmund singular integral operators and the decomposition in scales of their kernels that we are going to use in our analysis. Section \ref{sec:sec2} contains the proof of Theorem \ref{theoremA} via an iterative scheme. In Section \ref{sec:sec3} we analyze possible extensions of the main result to nonhomogeneous contexts, where partial results have been achieved, and the challenges that one must overcome to obtain a more satisfactory result than the existing ones. 

\noindent\textbf{Acknowledgements:} We thank the anonymous referees for a very careful reading of the manuscript, which led to a much better presentation of the material. Conde-Alonso also thanks Jill Pipher for initially suggesting him to write this note.

\section{Calder\'on-Zygmund decomposition, Whitney covering and singular integrals} \label{sec:sec1}

We use the standard dyadic system $\D$ of half-open cubes in $\R^d$ as a tool below, even if a martingale structure is not strictly necessary in order to prove a sparse domination result (see, for example, \cite{CADPPV2022}). Indeed, the sparse family $\Ss$ that appears in the statement of Theorem \ref{theoremA} is not a subfamily of $\D$.

\subsection{The Whitney covering theorem} The Whitney decomposition is a clever way to partition an open set $\Omega$ into dyadic cubes whose size is proportional to their distance to the boundary of $\Omega$. We will use the following version of the theorem, which has features that we shall use in our estimates below and the constants that will be needed below. To that end, define $k(d)$ as the smallest positive integer such that $2^{k(d)-1}> \sqrt{d}$, and set
\begin{equation*}
    \Lambda_d \coloneqq 2^{k(d)+2}+1.    
\end{equation*}

\begin{theorem}\label{Whitney}
Let $\Omega\subsetneq\R^d$ be open and fix a Whitney constant $C_{\mathcal{W}}>\sqrt{d}$. There exists a family $\mathcal{P}=\{L_j\} \subset \D$ such that
    \begin{enumerate}
        \item $\bigcup_{j}L_j=\Omega$, and the cubes $L_j$ are pairwise disjoint.
        \item $(C_{\mathcal{W}}-\sqrt{d})\cdot\ell(L_j)\leq \dist{L_j,\Omega^c}\leq 2C_{\mathcal{W}}\cdot \ell(L_j)$.
    \end{enumerate}
Moreover, if we take $C_{\mathcal{W}}>\sqrt{d}(2\Lambda_d+1)$, then
    \[
    2\Lambda_dL_j\subset \Omega  \text{ for all } j,
    \]
    and
    \[
    \text{if }\Lambda_dL_j\cap \Lambda_dL_k\neq\varnothing, \text{ then }\ell(L_j)\sim \ell(L_k).
    \]
\end{theorem}
The proof of the first two properties can be found in \cite[Ch. IV, Th. 1]{steinSingularIntegrals}, if we define the sets $\Omega_k$ that appear in that proof by 
$$
\Omega_{k}=\left\{x\in \R^d: C_{\mathcal{W}}2^{-k}<\dist{x,F}\leq C_{\mathcal{W}}2^{-k+1}\right\}.
$$
\subsection{Calder\'on-Zygmund decomposition with respect to a family of cubes} As above, let $\Omega$ be open and let $\mathcal{P}=\{L_j\}_{j\in J} \subset \D$ be a pairwise disjoint family such that $\bigcup L_j=\Omega$. Let $f\in L_{\mathrm{loc}}^{1}(\R^d)$. Define
\begin{align*}
    b_{L_j}
    &\coloneqq
    (f-\langle f\rangle_{L_j})\one_{L_j}, \quad \mbox{and} \quad b \coloneqq \sum_{j\in J}b_{L_j};
    \\
    g &\coloneqq
    f\one_{\Omega^c}+\sum_{j}\langle f\rangle_{L_j}\one_{L_j}.
    \end{align*}
The decomposition $f=g+b$ is called the Calder\'on-Zygmund decomposition of $f$ adapted to $\mathcal{P}$. Calder\'on-Zygmund decompositions are useful when the Calder\'on-Zygmund cubes are maximal with respect to the behavior of a certain maximal operator. We use the standard uncentered one:
$$
\mathcal{M}f(x) = \sup_{Q \ni x} \langle |f| \rangle_Q,
$$
with the supremum running over half-open cubes with sides parallel to the coordinate axes. We provide the details of the properties of the Calder\'on-Zygmund decomposition in this formulation.

\begin{lemma}\label{CZdecompProperties}
Let $\lambda>0$ be fixed and $0\leq f\in L_{\mathrm{loc}}^{1}(\R^d)$ be such that $\supp{f}\subset Q$. Let $\Omega \subset Q$ be open. Let $\{L_j\}$ be the Whitney decomposition of $\Omega$. If
    \begin{equation}\label{propiedadMaximalEnElAbierto}
        \mathcal{M}(f)\one_{\Omega^{c}}\lesssim \lambda,
    \end{equation}
where $\Omega^{c}=\R^d\setminus{\Omega}$, then    
\begin{equation}\label{inftyNormGoodPart}
        \|g\|_{\infty}\lesssim \lambda,
    \end{equation}
    \begin{equation}\label{L1NormBadPart}
        \sum_j \|b_{L_j}\|_{1}\lesssim \|f\|_{L^1(Q)},
    \end{equation}
    and if $R$ is a cube such that there exists $L_{j_0}\subset R$, then
    \begin{equation}\label{1.2_Jose}
        \langle |b|\rangle_{R}\lesssim\lambda.
    \end{equation}
\end{lemma}

\begin{proof}
We start proving \eqref{inftyNormGoodPart}. Since the sets $\Omega^{c}$ and $\{L_j\}_j$ are all pairwise disjoint, we first use \eqref{propiedadMaximalEnElAbierto} to get
\[
g\one_{\Omega^c}
=
f\one_{\Omega^c}
\leq
\mathcal{M}(f)\one_{\Omega^c}
\lesssim
\lambda,
\]
On the other hand, by Theorem \ref{Whitney}, there exists $\mu\sim 1$ such that $\mu L_j\cap \Omega^{c}\neq\varnothing$. Therefore, 
    \[
    \langle f\rangle_{L_j}
    \lesssim 
    \langle f\rangle_{\mu L_j}
    \leq 
    \mathcal{M}f(x_j)\quad \text{for some }x_j\in (\mu L_j)\cap\Omega^{c}.
    \]
By \eqref{propiedadMaximalEnElAbierto}, $\mathcal{M}f(x_j)\lesssim \lambda$, so $\langle f\rangle_{L_j}\lesssim \lambda$. Therefore, 
\[
g\one_{L_j}
=
\langle f\rangle_{L_j}\one_{L_j}
\lesssim 
\lambda.
\]
To prove \eqref{L1NormBadPart}, we just compute
\[
\|b_{L_j}\|_{L^1}
=
\|(f-\langle f\rangle_{L_j})\one_{L_j}\|_{L^1}
\leq 
\|f\one_{L_j}\|_{L^1}
+
\langle f\rangle_{L_j}\|\one_{L_j}\|_{L^1}
=
2\|f\one_{L_j}\|_{L^1}.
\]
Since the $\{L_j\}$ are pairwise disjoint and $\bigcup L_j=\Omega\subseteq Q$, we have
\[
\sum_j \|b_{L_j}\|_{1}
\leq 
\sum 2\|f\one_{L_j}\|_{L^1}
=
2\int_{\bigcup L_j} |f|
\leq
2\|f\|_{L^1(Q)}.
\]
To prove \eqref{1.2_Jose}, let $R$ be such that $L_{j_0}\subset R$ for some $j_0$. We have
\begin{align*}
    \frac{1}{|R|}\int_{R}|b(x)|\, dx
    &=
    \frac{1}{R}\int_{R}\left|\sum (f(x)-\langle f\rangle_{L_j})\one_{L_j}(x)\right|\, dx
    \\&\leq 
    \frac{1}{|R|}\sum_{j}
    \int_{R\cap L_j}|f(x)|\, dx
    +
    \frac{1}{|R|}\sum_{j}\langle f\rangle_{L_j}|R\cap L_j|
\end{align*}
Since $\mu L_{j}\cap \Omega^{c}\neq\varnothing$ for every $j$ and $L_{j_0}\subset R$, then $\mu R\cap \Omega^{c}\neq \varnothing$. So if $x_i\in \mu R\cap\Omega^{c}$, then 
\[
\frac{1}{|R|}\sum_{j}\int_{R\cap L_j}|f|
\leq
\frac{1}{|R|}\int_{R}|f|
\lesssim 
\frac{1}{|\mu R|}\int_{\mu R}|f|
\leq 
\mathcal{M}f(x_i)
\lesssim
\lambda.
\]
To bound the other sum, we just use that $\langle f\rangle_{L_j}\lesssim \lambda$ and the disjointness of $\{L_j\}$:
\[
\frac{1}{|R|}\sum_{j}\langle f\rangle_{L_{j}}|R\cap L_{j}|
\lesssim 
\frac{1}{|R|}\sum_{j}\lambda|R\cap L_{j}|
\leq  
\lambda\frac{1}{|R|}|R|
=
\lambda.    \qedhere
\]
\end{proof}

\subsection{Dini-smooth Calder\'on-Zygmund operators} A Calder\'on-Zygmund operator  (abbreviated CZO) is a linear operator $T$, initially defined over Schwartz functions, such that:
    \begin{itemize}
        \item $T$ can be extended to a bounded linear operator on $L^2(\mathbb{R}^d)$.
        \item For all $f\in L^\infty_c(\R^d)$, $Tf$ is given in integral form by a kernel $K\colon \mathbb{R}^d\times \mathbb{R}^d\setminus\{(x,y): x=y\}\to \mathbb{C}$, in the sense that
        \[
        Tf(x)
        = \int_{\mathbb{R}^d}K(x,y)f(y)\, dy, \quad \text{ for all } x\not\in \operatorname{supp}(f).
        \]
        \item $K$ satisfies the size condition
        $$
        |K(x,y)| \lesssim \frac{1}{|x-y|^d}, \quad x\neq y.
        $$
    \end{itemize}
We assume additional smoothness on $K$: a modulus of continuity $\omega$ is a function $\omega\colon [0,1]\to \mathbb{R}^{+}$ which is continuous and non-decreasing. We say that it satisfies the Dini condition if 
$$
\int_{0}^{1}\frac{\omega(t)}{t}\, dt<\infty.
$$
We require that there exists such a modulus of continuity $\omega$ so that, if $2|x-x'|\leq |x-y|$, then
\[
|K(x,y)-K(x',y)|+|K(y,x)-K(y,x')|\lesssim \frac{1}{|x-y|^d}\omega\left(\frac{|x-x'|}{|x-y|}\right).
\]
We break $K$ into scales using a standard Littlewood-Paley type decomposition. Let $\phi\colon \R^d\to[0,1]$ be a radial function in $C_c^{\infty}(\R^d)$ with $\supp{\phi}\subset \left\{x\in \mathbb{R}^d: 0<|x|<2\right\}$ and such that $\phi\equiv 1$ on $\{x \in \mathbb{R}^d: 0 \leq|x| \leq 1\}$. We set $\psi(x)\coloneqq \phi(x)-\phi(2 x)$, so that $\supp{\psi}\subset \left\{x\in \mathbb{R}^d: 1/2<|x|<2\right\}$. The family $\left\{\operatorname{Dil}_{2^k}^{\infty} \psi\right\}_{k \in \mathbb{Z}}=\left\{\psi\left(\frac{\cdot}{2^k}\right)\right\}_{k \in \mathbb{Z}}$ forms a partition of unity, in the sense that
\begin{equation*}
\sum_{k \in \mathbb{Z}} \psi\left(\frac{x}{2^k}\right)=1, \quad x \neq 0 . 
\end{equation*}
We next truncate $K$ using $\psi$. Set
$$
K_s(x,y) \coloneqq \psi\left(\frac{x-y}{2^s}\right)K(x,y), \quad s\in \Z,
$$
and define the smooth single-scale operator
$$
T_{s_0}f(x) \coloneqq \int_{\R^d}K_{s_0}(x,y)f(y)\, dy,\quad x\in\R^d,   
$$
and the smooth multi-scale truncations
$$
T_{(s_0)}f(x)\coloneqq \int_{\R^d}\left(\sum_{s\geq s_0} K_s(x,y)\right)f(y)\, dy, \quad T^{(s_0)}f \coloneqq Tf-T_{(s_0+1)}f.
$$
The integrals defining $T_{(s_0)}f$ and $T_s f$ converge absolutely for every $f\in L^p(\R^d)$, $1\leq p<\infty$, by Hölder's inequality. Therefore, the truncations are well-defined in the pointwise sense. For each $s_0\in \Z$, the operators $T_{s_0}$, $T_{(s_0)}$ and $T^{(s_0)}$ are Calder\'on-Zygmund operators with kernels $K_{s_0}(x,y)$, $\sum_{s\geq s_0}K_s(x,y)$ and $\sum_{s\leq s_0}K_s(x,y)$ respectively, with Dini smoothness given by the modulus of continuity $\widetilde{\omega}(t)\coloneqq t+\omega(t)$. Moreover, all these operators are bounded in $L^2(\R^d)$ uniformly with respect to $s_0\in\Z$. The adjoint operator of $T_s$ is given by
\[
(T_s)^{\ast}g(x)
\coloneqq 
\int_{\R^d}K_{s}(y,x)g(y)\, dx.
\]
Therefore, $(T_s)^{\ast}$ has the same structure as $T_s$. $T_s$ satisfies the $L^p$-improving-type inequality
\begin{equation}\label{singleScale1infty}
\|T_s(f)\|_{\infty} \lesssim_{d} 2^{-sd}\|f\|_{L^1}.    
\end{equation}
Indeed, if $x\in\R^d$, we can apply the size condition to obtain
\[
|T_s(f)(x)|
=
\left|
\int_{\R^d}K_s(x,y)f(y)\, dy
\right|
\leq 
\int_{|x-y|\in [2^{s-1},2^{s+1}]}\frac{|f(y)|}{|x-y|^d}\, dy
\lesssim 
2^{-sd}\|f\|_{L^1}.
\]
The same estimate is valid for $(T_s)^{\ast}$. We end this section with two important localization estimates. Given $Q_0 \in \D$, define $s_{Q_0} =\log_2(\ell(Q_0))$ and denote $S(Q_0) = s_{Q_0} + k(d)$. For $s > S(Q_0)$, we define the annulus
\begin{equation*}
A_{s}(Q_0) \coloneqq
\left\{x\in\R^d: |x-x_{Q_0}|\in [2^{-2}\cdot 2^{s},2^{k(d)}\cdot 2^{s}]
\right\},
\end{equation*}
where $x_Q$ denotes, here and throughout, the center of a cube $Q$.

\begin{lemma}\label{lemmaAnnuliDecomposition}
    Let $Q_0\in\D$ and $s>S(Q_0)$. Then 
    \begin{equation}\label{annuliProperty1}
        \text{if }\quad K_s(x,y)\neq 0\quad \text{ for }\quad y\in Q_0,\quad \text{ then }\quad  x\in A_s(Q_0).
    \end{equation}
    As a consequence, $\int_{Q_0} K_s(x,y)f(y)\, dy\neq 0$ only if $x\in A_s(Q_0)$. Therefore
    \begin{equation}\label{annuliDecomposition}
    Tf(x)
    =
    T^{(S(Q_0))}f(x)
    +
    \sum_{s>S(Q_0)}
    \left(\int_{Q_0} K_s(x,y)f(y)\, dy\right) \one_{A_{s}(Q_0)}(x).
    \end{equation}
\end{lemma}

\begin{proof}
    Let $s>S(Q_0)$, that is, $s=s_{Q_0}+k(d)+\delta$ for some integer $\delta\geq 1$. Notice that $K_s(x,y)\neq 0$ implies that $|x-y|\in [2^{s-1},2^{s+1}]$.
    Notice also that, since $y\in Q_0$, then 
    \[
    |x_{Q_0}-y|
    \leq
    \frac{\diam{Q_0}}{2}
    =
    \frac{2^{s_{Q_0}}\sqrt{d}}{2}
    <
    2^{s_{Q_0}+k(d)-2}.
    \]
    Therefore, we obtain the upper bound
    \begin{align*}
        |x-x_{Q_0}|
        \leq 
        |x-y|+|y-x_{Q_0}|
        < 
        2^{s+1}+2^{s_{Q_0}+k(d)-2}
        \leq
        2^{s_{Q_0}+k(d)+\delta}(2+2^{-2})
        \leq 
        2^{s}\cdot 2^{k(d)},
    \end{align*}
    where the last step is true because $k(d)\geq 2$ for $d\geq 1$. We also obtain the lower bound
    \begin{align*}
        |x-x_{Q_0}|
        \geq 
        |x-y|-|x_{Q_0}-y|
        \geq 
        2^{s-1}-2^{s_{Q_0}+k(d)-2}
        =
        2^{s_{Q_0}+k(d)+\delta}[2^{-1}-2^{-2-\delta}]
        \geq 
        2^{s}\cdot 2^{-2}.
    \end{align*}
\end{proof}

\begin{lemma}\label{localizationTs}
    Let $f\in L^1(\R^d)$ and $Q\in \D$. Then 
    \begin{equation}\label{localizedSupport}
    \supp{T^{(S(Q))}(f\one_{Q})}
    \subset 
    \Lambda_dQ.    
    \end{equation}
\end{lemma}

\begin{proof}
    Let $x\not\in \alpha Q$ for some $\alpha>1$. Then, $x\not\in\supp{f\one_Q}$, so
    \[
    T^{(s_Q)}(f\one_{Q})(x)
    =
    T(f\one_{Q})(x)
    -
    T_{(s_Q)}(f\one_{Q})(x)
    =
    \int_{Q}\left(\sum_{s\leq s_Q}K_s(x,y)\right)f(y)\, dy.
    \]
Since $K_s(x,y)=\psi(\frac{x-y}{2^s})K(x,y)$ and $\supp{\psi(\frac{\cdot}{2^s})}\subset [2^{s-1},2^{s+1}]$, we want to choose $\alpha$ in such a way that we guarantee $|x-y|>2^{S(Q)+1}$, so that $|x-y|>2^{s+1}$ for every $s\leq S(Q)$ and hence $T^{(S(Q))}(f\mathds{1}_{Q})(x)=0$. For every $y\in Q$ and $x\not\in\alpha Q$, we have 
    \[
    |x-y|
    \geq
    \frac{1}{2}\ell(\alpha Q)-\frac{1}{2}\ell(Q)
    =
    \frac{1}{2}\ell(Q)[\alpha-1]
    =
    2^{s_Q-1}[\alpha-1].
    \]
So to assure that $|x-y|\geq 2^{S(Q)+1}$ for every $x\not\in \alpha Q$ and $y\in Q$, it is enough that
    \[
    2^{s_Q-1}[\alpha-1]
    \geq
    2^{S(Q)+1}
    \iff 
    \alpha-1
    \geq
    2^{k(d)+2}
    \iff 
    \alpha 
    \geq
    2^{k(d)+2}+1
    =
    \Lambda_d.\qedhere
    \]
\end{proof}

\section{Sparse domination} \label{sec:sec2}

\subsection{The main estimate} The proof of Theorem \ref{theoremA} is based on the following lemma, that will be applied in an iterative way.

\begin{lemma}\label{lemma1.1_Jose}
Let $Q\in \D$. Let $0\leq f_1,f_2\in L^\infty(\R^d)$ with $\supp{f_1}\subset Q$ and $\supp{f_2}\subset \Lambda_dQ$. Let $\mathcal{P}=\left\{L_j\right\}_{j\in J}$ be the family of Whitney cubes of some open set $\Omega\subset 2\Lambda_dQ$ chosen with Whitney constant $C_{\mathcal{W}}>\sqrt{d}(2\Lambda_d+1)$, and suppose that
    \[
    (\mathcal{M}f_i) \one_{\Omega^c} \lesssim \left\langle f_i\right\rangle_{\Lambda_dQ},\quad i=1,2.
    \]
Then
    \[
    \left|\left\langle T^{(S(Q))} (f_1), f_2\right\rangle\right|
    \leq
    C\left\langle f_1\right\rangle_{\Lambda_dQ}\left\langle f_2\right\rangle_{\Lambda_dQ}|\Lambda_dQ|
    +
    \sum_{j\in J}\left|\left\langle T^{(S(L_j))} (f_1 \one_{L_j}), f_2\right\rangle\right|.
    \]
\end{lemma}

\begin{proof}
    Let 
    \[
    T_{\mathcal{P}}f(x)
    \coloneqq 
    \sum_{j\in J} T^{(S(L_j))}(f\one_{L_j}).
    \]
    Then we can split
    \begin{align*}
        \langle T^{(S(Q))}(f_1),f_2\rangle 
        &=
        \langle T^{(S(Q))}(f_1)-T_{\mathcal{P}}(f_1),f_2\rangle 
        +
        \langle T_{\mathcal{P}}(f_1),f_2\rangle 
        \\&=
        \langle T^{(S(Q))}(f_1)-T_{\mathcal{P}}(f_1),f_2\rangle 
        +
        \sum_{j\in J}\langle T^{(S(L_j))}(f_1\one_{L_j}),f_2\rangle.
    \end{align*}
    It is enough to prove the bound
    \[
    |\langle T^{(S(Q))}(f_1)-T_{\mathcal{P}}(f_1),f_2\rangle|
    \lesssim 
    \langle f_1\rangle_{\Lambda_dQ}\langle f_2\rangle_{\Lambda_dQ}|\Lambda_dQ|.
    \]
    We split the left-hand side using the Calder\'on-Zygmund decomposition from Lemma \ref{CZdecompProperties}, applied to each function $f_i$ and the collection $\mathcal{P}$ (corresponding to a level $\lambda_i= \langle f_i\rangle_{\Lambda_dQ}$). By the linearity of the operators involved, we get
    \begin{align*}
        \langle (T^{(S(Q))}-T_{\mathcal{P}})(f_1),f_2\rangle
        &=
        \langle (T^{(S(Q))}-T_{\mathcal{P}})(g_1),g_2\rangle
        +
        \langle (T^{(S(Q))}-T_{\mathcal{P}})(g_1),b_2\rangle
        \\&\phantom{=}+
        \langle (T^{(S(Q))}-T_{\mathcal{P}})(b_1),g_2\rangle
        +
        \langle (T^{(S(Q))}-T_{\mathcal{P}})(b_1),b_2\rangle
        \\&\eqqcolon 
        \mathrm{I}+\mathrm{II}+\mathrm{III}+\mathrm{IV}.
    \end{align*}

We estimate the four terms in turn.

\noindent \textit{Estimate of $\mathrm{I}$:} First we split
\[
    |\mathrm{I}|
    \leq
    |\langle T^{(S(Q))}(g_1),g_2\rangle|
    +
    |\langle T_{\mathcal{P}}(g_1),g_2\rangle|
    \leq 
    |\langle T^{(S(Q))}g_1,g_2\rangle|
    +
    \sum_{j}|\langle T^{(S(L_j))}(g_1\one_{L_j}),g_2\rangle|.
\]
For the first term, we use Cauchy-Schwarz, the uniform $L^2$-boundedness of $\{T^{(s)}\}_{s\in\Z}$ and \eqref{inftyNormGoodPart}:
\begin{align*}
|\langle T^{(s_Q)}(g_1),g_2\rangle|
\lesssim_{T}
\|g_1\|_2\|g_2\|_2 
\leq
\|g_1\|_{\infty}|\Lambda_dQ|^{\frac{1}{2}} \|g_2\|_{\infty}|\Lambda_dQ|^{\frac{1}{2}}
\lesssim
\langle f_1\rangle_{\Lambda_dQ}\langle f_2\rangle_{\Lambda_dQ}|\Lambda_dQ|.
\end{align*}
    For the second one, using the localization of Lemma \ref{localizationTs}, we have
    \[
    |\langle T^{(S(L_j))}(g_1\one_{L_j}),g_2\rangle|
    =
    |\langle T^{(S(L_j))}(g_1\one_{L_j}),g_2\one_{\Lambda_dL_j}\rangle|,
    \]
    and so
    \[
    |\langle T^{(S(L_j))}(g_1\one_{L_j}),g_2\one_{\Lambda_dL_j}\rangle|
    \leq
    \|T^{(S(L_j))}(g_1\one_{L_j})\|_{2}\|g_2\one_{\Lambda_dL_j}\|_{2}
    \lesssim 
    \langle f_1\rangle_{\Lambda_dQ}\langle f_2\rangle_{\Lambda_dQ}|L_j|,
    \]
    arguing as before. Since $\{L_j\}$ are disjoint with $\bigcup L_j=\Omega\subseteq 2\Lambda_dQ$, we have
    \begin{align*}
        \sum_{j}|\langle T^{(S(L_j))}(g_1\one_{L_j}),g_2\rangle|
        &\lesssim 
        \sum_{j}\langle f_1\rangle_{\Lambda_dQ}\langle f_2\rangle_{\Lambda_dQ}|L_j|
        \lesssim
        \langle f_1\rangle_{\Lambda_dQ}\langle f_2\rangle_{\Lambda_dQ}|\Lambda_dQ|.
    \end{align*}

\noindent \textit{Estimate of $\mathrm{II}$:} Since $g_1\one_{L_j}=\langle f_1\rangle_{L_j}$, we have
    \begin{align*}
        \mathrm{II}
        &=
        \langle (T^{(S(Q))}-T_{\mathcal{P}})(g_1),b_2\rangle 
        \\&=
        \langle T^{(S(Q))}(f_1\one_{\Omega^{c}}),b_2\rangle 
        +
        \langle T^{(S(Q))}(g_1\one_{\bigcup L_j}),b_2\rangle 
        -
        \langle T_{\mathcal{P}}(g_1),b_2\rangle
        \\&=
        \langle T^{(S(Q))}(f_1\one_{\Omega^{c}}),b_2\rangle 
        +
        \sum_{j}\langle (T^{(S(Q))}-T^{(S(L_j))})(\langle f_1\rangle_{L_j}\one_{L_j}),b_2\rangle
        \\&\eqqcolon 
        \mathrm{II}_1+\mathrm{II}_2.
    \end{align*}
    We begin with $\mathrm{II}_1$. To use the smoothness of the kernel, we pass to the adjoint to take advantage of the mean $0$ atoms $b_{i,L_j}$. We have 
    \[
    \mathrm{II}_1
    =
    \langle f_1\one_{\Omega^{c}},(T^{(S(Q))})^{\ast}(b_2)\rangle 
    =
    \sum_{j}\int_{\Omega^c}f_1(x)\left(\int_{\R^d}\left(\sum_{s\leq S(Q)}K_s(y,x)\right)b_{2,L_j}(y)\, dy\right)\, dx,
    \]
    using the integral expression of $T$. Since $C_{\mathcal{W}}>\sqrt{d}(2\Lambda_d+1)=\sqrt{d}(2^{k(d)+3}+2)$, 
        \[
        \dist{L_j,\Omega^c}\geq (C_{\mathcal{W}}-\sqrt{d})\ell(L_j)>2^{S(L_j)+1}.
        \]
        This means that $|x-y|>2^{S(L_j)+1}$ for every $x\in\Omega^c$ and $y\in L_j$. By definition, $K_s(x,y)=0$ for all $s\leq S(L_j)$ and $ x\in\Omega^c$, $y\in L_j$. Therefore, the sum in $\mathrm{II}_1$ can truncated accordingly, and due to the support and mean $0$ properties of $b_{2,L_j}$, we obtain
        \begin{align}\label{estimateII1_formula1}
        \nonumber|\mathrm{II}_1|
        &\leq 
        \sum_j\sum_{s=S(L_j)+1}^{S(Q)}\int_{\Omega^c}|f_1(x)|\left(\int_{L_j}|K_s(y,x)-K_s(x_{L_j},x)|\cdot |b_{2,L_j}(y)|\, dy\right)\, dx
        \\\nonumber &=
        \sum_j\sum_{s=S(L_j)+1}^{S(Q)}\left(\int_{L_j}|b_{2,L_j}(y)|\, dy\right)\left(\int_{\Omega^{c}}|f_1(x)|\sup_{z\in L_j}|K_s(z,x)-K_s(x_{L_j},x)|\, dx\right)
        \\&\lesssim 
        \langle f_1\rangle_{\Lambda_dQ}
        \sum_j \|b_{2,L_j}\|_{L^1}\sum_{s=S(L_j)+1}^{S(Q)}\int_{\Omega^c\cap A_{s}(L_j)}\sup_{z\in L_j}|K_s(z,x)-K_s(x_{L_j},x)|\, dx,
        \end{align}
having used in the last step that, since $s>S(L_j)$, then $K_s(z,x)$ can only be non-zero for $z\in L_j$ if $x\in A_{s}(L_j)$ by formula \eqref{annuliProperty1} in Lemma \ref{lemmaAnnuliDecomposition}. To bound the remaining sum of integrals, we use the smoothness of the kernels $K_s$. Indeed, $K_s$ has Dini smoothness with modulus of continuity given by $\Tilde{\omega}(t)=t+\omega(t)$. Therefore, for $x\in A_s(L_j)$ and $z\in L_j$,
    \begin{equation}\label{chekingSmoothnessLemma}
    |K_s(x_{L_j},x)-K_s(z,x)|
    \lesssim 
    \frac{1}{|x-x_{L_j}|^d}
    \left[
    \frac{|z-x_{L_j|}}{|x-x_{L_j}|}+\omega\left(\frac{|z-x_{L_j|}}{|x-x_{L_j}|}\right)
    \right],
    \end{equation}
applying the smoothness condition because $|x-x_{L_j}|\geq 2|z-x_{L_j}| $. Moreover, on the one hand $x\in A_s(L_j)$ implies $|x-x_{L_j}|\geq 2^{s-2}$. On the other hand, $|x_{L_j}-z| \lesssim_d \ell(L_j)$ if $z \in L_j$, so 
    \begin{equation}\label{smoothnessLj}
    |K_s(x_{L_j},x)-K_s(z,x)|
    \lesssim 
    2^{-d[s-2]}
    \left[\frac{\ell(L_j)}{2^{s-2}}+\omega\left(\frac{\ell(L_j)}{2^{s-2}}\right)\right].
    \end{equation}
\eqref{smoothnessLj} allows us to sum up:
        \begin{align*}
        |\mathrm{II}_1|
        &\lesssim
        \langle f_1\rangle_{\Lambda_dQ}\sum_j\|b_{2,L_j}\|_{1}\sum_{s=S(L_j)+1}^{S(Q)}|A_s(L_j)| 2^{-d[s-2]}
    \left[\frac{\ell(L_j)}{2^{s-2}}+\omega\left(\frac{\ell(L_j)}{2^{s-2}}\right)\right] \\
    & \lesssim \langle f_1\rangle_{\Lambda_dQ}\sum_j\|b_{2,L_j}\|_{1}\sum_{s=S(L_j)+1}^{S(Q)}2^{sd} 2^{-d[s-2]}
    \left[2^{s_{L_j}-s}+\omega\left(2^{s_{L_j}-s-2}\right)\right] \\
    & \lesssim 
        \langle f_1\rangle_{\Lambda_dQ}\sum_j\|b_{2,L_j}\|_{1}\left[\sum_{s=S(L_j)+1}^{S(Q)}
        2^{s_{L_j}-s}
        +
        \sum_{s=S(L_j)+1}^{S(Q)}\omega\left(2^{s_{L_j}-s-2}\right)
        \right]
        \\&\leq 
        \langle f_1\rangle_{\Lambda_dQ}\sum_j\|b_{2,L_j}\|_{1}\left[1+\int_{0}^{1}\frac{\omega(t)}{t}\, dt\right], 
        \end{align*}
using $|A_s(L_j)| \sim 2^{sd}$. So, by \eqref{L1NormBadPart},
        \begin{align*}
        |\mathrm{II}_1|    
        \lesssim_{\omega}
        \langle f_1\rangle_{\Lambda_dQ}\sum_j\|b_{2,L_j}\|_{L^1}
        \lesssim
        \langle f_1\rangle_{\Lambda_dQ}\|f_2\|_{L^1(\Lambda_dQ)}
        =
        \langle f_1\rangle_{\Lambda_dQ}\langle f_2\rangle_{\Lambda_dQ}|\Lambda_dQ|.
    \end{align*}
We now estimate $\mathrm{II}_2$. As for $\mathrm{II}_1$, to use the smoothness of the kernel we pass to the adjoint operator:
\begin{align*}
\mathrm{II}_2 & = \sum_{j}\langle \langle f_1\rangle_{L_j}\one_{L_j},\left(T_{(S(L_j)+1)}^{(S(Q))}\right)^{\ast}(b_2)\rangle \\
& = \sum_j\sum_{s=S(L_j)+1}^{S(Q)}\langle \langle f_1\rangle_{L_j}\one_{L_j},(T_s)^{\ast}(b_2)\rangle \\
& = \sum_{j}\sum_{s=S(L_j)+1}^{S(Q)}\sum_{j'}
\langle \langle f_1\rangle_{L_j}\one_{L_j},(T_s)^{\ast}(b_{2,L_{j'}})\rangle.
\end{align*}
To apply the smoothness of $K_s$ to $b_{2,L_{j'}}$ we need $s\geq S(L_{j'})+1$ as before, so we split the sum as
\begin{equation}\label{II_2LoQueHayQueEstimar}
\mathrm{II}_2
=
\sum_{j}\langle f_1\rangle_{L_j}\sum_{j'}\sum_{s=S(L_j)+1}^{S(L_{j'})}\langle \one_{L_j},(T_s)^{\ast}(b_{2,L_{j'}})\rangle
+
\sum_{j}\langle f_1\rangle_{L_j}\sum_{j'}\sum_{s=S(L_{j'})+1}^{S(Q)}\langle \one_{L_j},(T_s)^{\ast}(b_{2,L_{j'}})\rangle,    
\end{equation}
where the first sum is considered empty when $S(L_{j})+1>S(L_{j'})$, i.e., when $s_{L_j}+1>s_{L_{j'}}$. Now, for the second term all computations are the same as for $\mathrm{II}_1$. We just need to notice that we can get rid off the sum in $j$ because the $\{L_j\}$ are disjoint:
\[
\sum_j\int_{L_j\cap A_s(L_{j'})}\left|\int_{L_{j'}}K_s(y,x)b_{2,L_{j'}}(y)\, dy\right|\, dx
\leq 
\int_{A_s(L_{j'})}\left|\int_{L_{j'}}K_s(y,x)b_{2,L_{j'}}(y)\, dy\right|\, dx.
\]
Therefore, we are left to estimate the first sum in \eqref{II_2LoQueHayQueEstimar} when $s_{L_{j'}}>s_{L_j}$.
But the number of terms that we need to estimate is controlled by a constant that depends only on $C_{\mathcal{W}}$ and $d$. Indeed, we next check that the scales of $L_j$ and $L_{j'}$ have to be comparable in order for the corresponding term to be non-zero. Suppose that $s\in \Z$ is such that $S(L_j)+1\leq s\leq S(L_{j'})$, with $s_{L_j}<s_{L_{j'}}$, and that $\langle \one_{L_j},(T_s)^{\ast}(b_{2,L_{j'}})\rangle\neq 0$.
In particular, $\operatorname{supp}(\one_{L_j})\cap\operatorname{supp}(T_s)^{\ast}(b_{2,L_{j'}})\neq \varnothing.$
Recall that $\supp{\one_{L_j}}\subset L_j$ and, by (the proof of) Lemma \ref{localizationTs}, $\operatorname{supp}(T_s)^{\ast}(b_{2,L_{j'}})\subset\Lambda_dL_{j'}$ for $s\leq S(L_{j'})$.
So $L_j\cap \Lambda_dL_{j'}\neq\varnothing$. Therefore,
\[
\dist{L_{j'},\Omega^c}
\leq 
\diam{\Lambda_dL_{j'}}
+
\diam{L_{j}}
+
\dist{L_j,\Omega^c}.    
\]
Then, by Theorem \ref{Whitney} and using $C_{\mathcal{W}}>\sqrt{d}(2\Lambda_d+1)$ we get
\[
\dist{L_{j'},\Omega^c}
\geq 
(C_{\mathcal{W}}-\sqrt{d})2^{s_{L_{j'}}}
\]
and
\begin{align*}
\diam{\Lambda_dL_{j'}}
+
\diam{L_j}
+
\dist{L_{j},\Omega^c}
&=
\sqrt{d}\Lambda_d\ell(L_{j'})
+
\sqrt{d}\ell(L_j)
+
\dist{L_j,\Omega^c}
\\&\leq 
\sqrt{d}\Lambda_d2^{s_{L_{j'}}}
+
\sqrt{d}2^{s_{L_j}}
+
2C_{\mathcal{W}}2^{s_{L_j}}.    
\end{align*}
Therefore,
\[
[C_{\mathcal{W}}-\sqrt{d}(1+\Lambda_d)]2^{s_{L_{j'}}}
\leq 
[\sqrt{d}+2C_{\mathcal{W}}]2^{s_{L_j}},
\]
which implies
\[
\ell(L_{j'})
\lesssim_{C_{\mathcal{W}},d} 
\ell(L_j).
\]
Since $s_{L_j}<s_{L_{j'}}$, there is $N_0$ such that $\ell(L_j) < \ell(L_{j'})\leq 2^{N_0}\ell(L_j)$. Additionally, since $L_j\cap \Lambda_dL_{j'}\neq \varnothing$ with $\ell(L_j)\sim \ell(L_{j'})$, there exists $\mu>0$ (independent of $j$ and $j'$) such that $L_{j'}\subset \mu L_{j}$. We also have the $L^1\to L^\infty$ estimate \eqref{singleScale1infty} for $(T_s)^{\ast}$, which yields
\[
|\langle \one_{L_j},(T_s)^{\ast}(b_{2,L_{j'}})\rangle|
\leq 
\|\one_{L_j}\|_{1}
\|(T_s)^{\ast}(b_{2,L_{j'}})\|_{\infty}
\lesssim 
|L_j|\cdot
2^{-sd}\|b_{2,L_{j'}}\|
\lesssim 
|L_j|\cdot2^{-sd}\|f_2\one_{L_{j'}}\|_{L^1}.
\]
We plug it in the estimate that remains to get
\begin{align*}
\left|\sum_{j}\langle f_1\rangle_{L_j}\sum_{j'} \right. & \left. \sum_{s=S(L_j)+1}^{S(L_{j'})}\langle \one_{L_j},(T_s)^{\ast}(b_{2,L_{j'}})\rangle \right| \\
& \lesssim \langle f_1\rangle_{\Lambda_d Q}  \sum_{\substack{j,j'\\ \ell(L_j)<\ell(L_{j'})\leq 2^{N_0} \ell(L_j)\\ L_{j'}\subset \mu L_j}}
\sum_{s=S(L_j)+1}^{S(L_{j})+N_0+1}
\left|\langle \one_{L_j},(T_s)^{\ast}(b_{2,L_{j'}})\rangle\right| \\
& \lesssim \langle f_1\rangle_{\Lambda_d Q}  \sum_{\substack{j,j'\\ \ell(L_j)<\ell(L_{j'})\leq 2^{N_0} \ell(L_j)\\ L_{j'}\subset \mu L_j}}
\sum_{s=S(L_j)+1}^{S(L_{j})+N_0+1}
|L_j|\cdot2^{-sd}\|f_2\one_{L_{j'}}\|_{L^1} \\
& \lesssim \langle f_1\rangle_{\Lambda_d Q}  \sum_{\substack{j,j'\\ \ell(L_j)<\ell(L_{j'})\leq 2^{N_0} \ell(L_j)\\ L_{j'}\subset \mu L_j}}
\sum_{s=S(L_j)+1}^{S(L_{j})+N_0+1}
|L_{j}|\langle f_2\rangle_{L_{j'}}\\
& \lesssim \langle f_1\rangle_{\Lambda_d Q} \langle f_2\rangle_{\Lambda_d Q} \sum_j |L_{j}|\sum_{\substack{j'\\ \ell(L_j)<\ell(L_{j'})\leq 2^{N_0} \ell(L_j)\\ L_{j'}\subset \mu L_j}} \sum_{s=S(L_j)+1}^{S(L_{j})+N_0+1} 1\\
& \lesssim \langle f_1\rangle_{\Lambda_d Q}
\langle f_2\rangle_{\Lambda_d Q}|\Lambda_d Q|.
\end{align*}
The last step above is justified by the fact that the $L_j$ are pairwise disjoint and contained in $2\Lambda_d Q$, and because the two inner sums run over boundedly many terms: the innermost runs over $N_0$ terms, while $\mu$ just depends on $C_{\mathcal{W}}$ and $d$, so
\[
\#\{j': \ell(L_j)<\ell(L_{j'}),\hspace{0.1cm} L_{j'}\subset \mu L_j\}
\lesssim_{C_{\mathcal{W}},d}
1.
\]

\noindent \textit{Estimate of $\mathrm{III}$:} By definition and the $L^2$-continuity and linearity of the operators and the scalar product,
    \[
    \mathrm{III}
    =
    \langle (T^{(S(Q))}-T_{\mathcal{P}})(b_1),g_2\rangle
    =
    \sum_{j}\langle T^{(S(Q))}(b_{1,L_j}),g_2\rangle 
    -
    \sum_{j}\sum_{j'}T^{(S(L_{j'}))}(b_{1,L_j}\one_{L_{j'}}),g_2\rangle.
    \]
    Since $\supp{b_{1,L_j}}\subset L_j$ and $L_{j}\cap L_{j'}=\varnothing$ if $j\neq j'$, then $b_{1,L_j}\one_{L_{j'}}\equiv 0$ if $j\neq j'$. So the above sum is
    \[
    \sum_j \langle (T^{(S(Q))}-T^{(S(L_j))})(b_{1,L_j}),g_2\rangle 
    =
    \sum_j\sum_{s=S(L_j)+1}^{S(Q)}\langle T_s(b_{1,L_j}),g_2\rangle.
    \]
    We have the scalar product of a good function against the operators $T_s$ with $s>S(L_j)$ acting on the mean $0$ functions $b_{1,L_j}$ 
    Therefore, the estimate is exactly the same as in $\mathrm{II}$, with the only differences that the indices $i=1,2$ are interchanged and we have $T_s$ instead of $(T_s)^{\ast}$. Hence,
    \[
    |\mathrm{III}|
    \lesssim 
    \langle f_1\rangle_{\Lambda_d Q}
    \langle f_2\rangle_{\Lambda_d Q}
    |\Lambda_d Q|.
    \]
\noindent \textit{Estimate of $\mathrm{IV}$:} By definition of $T_{\mathcal{P}}$ and $b_1$, and continuity and linearity of the operators and the scalar product,
    \[
    \mathrm{IV}
    =
    \langle (T^{(S(Q))}-T_{\mathcal{P}})(b_1),b_2\rangle 
    =
    \sum_j \langle T^{(S(Q))}(b_{1,L_j}),b_2\rangle
    -
    \sum_j\sum_{j'}\langle T^{(S(L_{j'}))}(b_{1,L_j}\one_{L_{j'}}),b_2\rangle.
    \]
    As reasoned in $\mathrm{III}$, $b_{1,L_j}\one_{L_{j'}}\equiv 0$ if $j\neq j'$. So
    \[
    \mathrm{IV}
    =
    \sum_j \langle (T^{(S(Q))}-T^{(S(L_j))}(b_{1,L_j}),b_2\rangle.
    \]
    Therefore, the estimate is of the same form as $\mathrm{III}$, with the small difference that now we have $b_2$ instead of $g_2$. We write the details to show the minor changes that need to be made in the argument.
    
    \begin{align*}    
    \mathrm{IV}
    &=   
    \sum_j \int_{\R^d}\left(\int_{\R^d}\left(
    \sum_{s=S(L_j)+1}^{S(Q)} K_s(x,y)
    \right)b_{1,L_j}(y)\, dy\right) b_2(x)\, dx
    \\&=
    \sum_{j}\sum_{s=S(L_j)+1}^{S(Q)} \int_{\R^d}\left(\int_{L_j}[K_s(x,y)-K_s(x,x_{L_j})]b_{1,L_j}(y)\, dy\right) b_2(x)\, dx.
    \end{align*}
    Introducing absolute values, we have
    \[
    |\mathrm{IV}|
    \leq 
    \sum_{j}
    \sum_{s=S(L_j)+1}^{S(Q)}
    \left(\int_{L_j}|b_{1,L_j}(y)|\, dy\right)\left(\int_{\R^d} \sup_{z\in L_j}|K_s(x,z)-K_s(x,x_{L_j})|b_2(x)|\, dx\right).
    \]
        
    Again, we have operators $T_s$ acting on $b_{1,L_j}$ with scales $s>S(L_j)$. So we can restrict the integral over $\R^d$ to $A_s(L_j)$ by Lemma \ref{lemmaAnnuliDecomposition} and apply the smoothness estimate as in $II_1$, obtaining
    \begin{align}\label{smoothnessIV}
    \nonumber&\int_{A_s(L_j)}\sup_{z\in L_j}|K_s(x,z)-K_s(x,x_{L_j})|\cdot |b_2(x)|\, dx
    \\&\lesssim 
    \int_{A_s(L_j)}\frac{1}{|x-x_{L_j}|^d}\left[\frac{|z-x_{L_j}|}{|x-x_{L_j}|}+\omega\left(\frac{|z-x_{L_j}|}{|x-x_{L_j}|}\right)\right]|b_2(x)|\, dx,   
    \end{align}
    using again $2|z-x_{L_j}|\leq |x-x_{L_j}|$ to apply the smoothness. 
    Since $x\in A_s(L_j)$, $|x-x_{L_j}|\geq 2^{s-2}$, and since $x_{L_j},z\in L_j$, $|z-x_{L_j}|\leq \sqrt{d}2^{s_{L_j}-1}\leq 2^{S(L_j)-2}$. So 
    \[
    \frac{1}{|x-x_{L_j}|^d}
    \lesssim 
    2^{-sd}
    \quad \text{ and } \quad
    \frac{|z-x_{L_j}|}{|x-x_{L_j}|}
    \leq 
    \frac{1}{2^{s-S(L_j)}}.
    \]
    Hence, using that $\omega$ is non-decreasing, \eqref{smoothnessIV} can be bounded by
    \[
    \eqref{smoothnessIV}
    \lesssim
    2^{-sd}\left[\frac{1}{2^{s-S(L_j)}}+\omega\left(\frac{1}{2^{s-S(L_j)}}\right)\right]
    \int_{A_s(L_j)}|b_2(x)|\, dx.
    \]
Denote $B_s(L_j)=B(x_{L_j};2^{s+k(d)+1})$ the smallest ball that contains $A_s(L_j)$. Then
    \[
    \int_{A_s(L_j)}|b_2(x)|\, dx
    \leq 
    \int_{B_s(L_j)}|b_2(x)|\, dx
    =
    |B_s(L_j)|\langle |b_2|\rangle_{B_s(L_j)}
    \sim 
    2^{sd}\langle |b_2|\rangle_{B_s(L_j)} \lesssim 2^{sd}\langle f_2\rangle_{\Lambda_d Q},
    \]
    by \eqref{1.2_Jose}. Therefore, 
    \begin{align*}
        |\mathrm{IV}|
        &\leq 
        \sum_{j}
        \sum_{s=S(L_j)+1}^{S(Q)}
        \left(\int_{L_j}|b_{1,L_j}(y)|\, dy\right)\left(\int_{\R^d} \sup_{z\in L_j}|K_s(x,z)-K_s(x,x_{L_j})|b_2(x)|\, dx\right)
        \\&\lesssim 
        \langle f_2\rangle_{\Lambda_d Q}\sum_{j}\|b_{1,L_j}\|_{L^1}\sum_{s=S(L_j)+1}^{S(Q)}\left[\frac{1}{2^{s-S(L_j)}}+\omega\left(\frac{1}{2^{s-S(L_j)}}\right)\right]
        \\&\lesssim 
        \langle f_2\rangle_{\Lambda_d Q}\sum_{j}\|b_{1,L_j}\|_{L^1}
        \lesssim 
        \langle f_2\rangle_{\Lambda_d Q}\|f_1\|_{L^1(\Lambda_d Q)}
        =
        \langle f_1\rangle_{\Lambda_d Q}\langle f_2\rangle_{\Lambda_d Q}|\Lambda_d Q|,
    \end{align*}
    having applied \eqref{L1NormBadPart} in the last inequality.\qedhere
\end{proof}
 
\subsection{Construction of the sparse family} We can now finish the proof of Theorem \ref{theoremA}. 

\begin{proof}[Proof of Theorem \ref{theoremA}] The idea of the proof is to apply Lemma \ref{lemma1.1_Jose} iteratively, for which we qualitatively assume that $f_1,f_2 \in L^\infty(\mathbb{R}^d)$. We will denote by $\mathcal{P}_n$ the family of cubes on which we will localise the estimate at the $n$-th iteration. We know that $\supp{f_1}\subset Q_0$, so we let $\mathcal{P}_0\coloneqq \{Q_0\}$.

Fix $\eta_0\in (0,1)$. We will build a $(C,\eta_0)$-sparse family such that the estimate in Theorem \ref{theoremA} holds, for some constant $C\in (0,1)$. We do it following an iteration procedure. For the first iteration, let 
    \[
    \Omega_i^{Q_0}
    \coloneqq 
    \{
    x\in\R^d: \mathcal{M}f_i(x)>c_0\langle f_i\rangle_{\Lambda_{d} Q_0}
    \},\quad i=1,2,
    \]
    for some $c_0$ to be chosen later. Define the open set
    \[
    \Omega_{Q_0}
    \coloneqq 
    \Omega_1^{Q_0}\cup \Omega_2^{Q_0},
    \]
which is contained in $2\Lambda_d Q_0$ and gives
\[
\mathcal{M}f_i(x)\leq c_0\langle f_i\rangle_{\Lambda_d Q_0}\quad \text{ if }x\in\Omega_{Q_0}^{c},
\]
for $i=1,2$. Let $\{L_j^{(1)}\}_{j\in J_1(Q_0)}$ be the Whitney decomposition of $\Omega_{Q_0}\subset 2\Lambda_dQ_0$ with Whitney constant $C_{\mathcal{W}}>\sqrt{d}(2\Lambda_d+1)$ as in Theorem \ref{Whitney}. Since $\supp{f_1}\cup \supp{f_2}\subset Q_0$, we can apply Lemma \ref{lemma1.1_Jose} (notice that $\langle Tf_1,f_2\rangle =\langle T^{(S(Q_0))}f_1,f_2\rangle$) and obtain the following estimate:
    \begin{equation}\label{eq:Level0Bound}
        \left|\left\langle Tf_1, f_2\right\rangle\right| = \left|\left\langle T^{(S(Q_0))} (f_1), f_2\right\rangle\right|
        \leq
        C\left\langle f_1\right\rangle_{\Lambda_dQ_0}\left\langle f_2\right\rangle_{\Lambda_dQ_0}|\Lambda_dQ_0|
        +
        \sum_{j\in J_1(Q_0)}\left|\left\langle T^{(S(L_j^{(1)}))} (f_1 \one_{L_j^{(1)}}), f_2\right\rangle\right|.
    \end{equation}
    Assume without loss of generality that $Q_0$ is dyadic. Since $\supp{f_1}\subset Q_0$, the non-zero terms in the sum will be those such that $Q_0\cap L_j^{(1)}\neq\varnothing$. By the choice of $C_{\mathcal{W}}$, $2\Lambda_d L_j\subset \Omega_{Q_0}\subset 2\Lambda_d Q_0$, so by dyadic nesting it must be $L_j^{(1)}\subset Q_0$. Hence, the non-zero terms in the sum correspond to the cubes 
    \[
    \mathcal{P}_1
    \coloneqq
    \{L_j^{(1)}: j\in J_1(Q_0), L_j^{(1)}\subset Q_0\}.
    \]
    If we denote $\mathcal{P}_1=\{L_j^1\}_{j\in J_1}$, estimate \eqref{eq:Level0Bound} can be written as
    \[
    \left|\left\langle T^{(S(Q))} (f_1), f_2\right\rangle\right|
    \leq
    C\left\langle f_1\right\rangle_{\Lambda_dQ_0}\left\langle f_2\right\rangle_{\Lambda_dQ_0}|\Lambda_dQ_0|
    +
    \sum_{j\in J_1}\left|\left\langle T^{(S(L_j^{1}))} (f_1 \one_{L_j^{1}}), f_2\right\rangle\right|.
    \]
    Now we repeat the process for each of the terms in the sum. Define
    \[
    \Omega_{i}^{L_{j}^{1}}
    \coloneqq 
    \{
    x\in\R^d: \mathcal{M}(f_i\mathds{1}_{\Lambda_d L_j^{1}})(x)
    >
    c_0\langle f_i\rangle_{\Lambda_d L_j^{1}}\},\quad i=1,2.
    \]
    Then $\Omega_{L_j^{1}}\coloneqq \Omega_{1}^{L_{j}^{1}}\cup \Omega_{2}^{L_{j}^{1}}$ is contained in $2\Lambda_d L_j^{1}$ and we have
    \[
    \mathcal{M}(f_i\mathds{1}_{\Lambda_d L_j^{1}})(x)
    \lesssim 
    \langle f_i\rangle_{\Lambda_d L_j^{1}} \quad \text{ if }x\in\Omega_{L_j^1}^{c}, \quad i=1,2.
    \]
    Taking the Whitney decomposition $\{L_{k}^{(2)}\}_{k\in J_2(j)}$ of $\Omega_{L_j^{1}}$, we can apply Lemma \ref{lemma1.1_Jose} and obtain 
    \begin{align}\label{eq:Level1Bound}
        \nonumber|\langle T^{(S(L_j^{1})}(f_1\mathds{1}_{L_j^{1}},f_2\rangle |
        &\leq 
        C\langle f_1\mathds{1}_{L_j^{1}}\rangle_{\Lambda_d L_j^{1}}\langle f_2\rangle_{\Lambda_d L_j^{1}}|\Lambda_d L_j^{1}|
        \\&\phantom{=}+
        \sum_{k\in J_2(j)}|\langle T^{(S(L_{j_2}^{(2)}))}(f_1\mathds{1}_{L_j^{1}}\mathds{1}_{L_{k}^{(2)}},f_2\rangle|.
    \end{align}
    for each $j\in J_1$. Again, for the terms in the last sum to be non-zero we need $L_j^{1}\cap L_{k}^{(2)}\neq \varnothing$. Since $2\Lambda_dL_{k}^{(2)}\subset \Omega_{L_j^{1}}\subset 2\Lambda_d L_j^{1}$, then $L_{k}^{(2)}\subset L_j^{1}$. Hence, the sum in \eqref{eq:Level1Bound} can be taken over the cubes in 
    \[
    \mathcal{P}_{2}(j)
    \coloneqq 
    \{L_{k}^{(2)}: j_2\in J_2(j) \text{ for some }j\in J_1 \text{ and }L_{k}^{(2)}\subset L_j^{1}\}, \quad j\in J_1.
    \]
    The collection of second generation cubes is
    \[
    \mathcal{P}_2
    \coloneqq 
    \bigcup_{j\in J_1}\mathcal{P}_2(j)
    \eqqcolon 
    \{L_k^{2}\}_{k\in J_2}.
    \]
$\mathcal{P}_1=\{L_j^{1}\}_{j\in J_1}$ is a pairwise disjoint family. Likely, the cubes in $\mathcal{P}_2$ are also pairwise disjoint. Indeed, let $L,L'\in \mathcal{P}_2$. This means that $L=L_k^{(2)}\in \mathcal{P}_2(j)$ and $L'=L_{k'}^{(2)}\in\mathcal{P}_2(j')$ for some $j,j'\in J_1$. If $j=j'$, then $L_k^{(2)}$ and $L_{k'}^{(2)}$ are disjoint because they are Whitney cubes of the open set $\Omega_{L_j^1}$. If $j\neq j'$, then $L_k^{(2)}$ and $L_{k'}^{(2)}$ are disjoint because $L_k^{(2)}\cap L_{k'}^{(2)}\subset L_{j}^1\cap L_{j'}^1=\varnothing$ because $L_j^{1}$ and $L_{j'}^{1}$ are Whitney cubes of $\Omega_{Q_0}$. We can now combine estimates \eqref{eq:Level0Bound} and \eqref{eq:Level1Bound} to obtain
    \begin{equation}\label{eq:Level2Bound}
        |\langle Tf_1,f_2\rangle|
        \leq 
        C\sum_{i=0,1}\sum_{Q\in\mathcal{P}_i}
        \langle f_1\rangle_{\Lambda_d Q}\langle f_2\rangle_{\Lambda_d Q}|\Lambda_d Q|
        +
        \sum_{L\in \mathcal{P}_2}|\langle T^{(S(L))}(f_1\mathds{1}_{L}),f_2\rangle|.
    \end{equation}
To perform the $n$-th iteration, $n\geq 3$, we have from the $(n-1)$-th iteration that
    \begin{equation}\label{eq:LevelNMinus1Bound}
        |\langle Tf_1,f_2\rangle|
        \leq 
        C\sum_{i=0}^{n-2}\sum_{Q\in\mathcal{P}_i}
        \langle f_1\rangle_{\Lambda_d Q}\langle f_2\rangle_{\Lambda_d Q}|\Lambda_d Q|
        +
        \sum_{L\in \mathcal{P}_{n-1}}|\langle T^{(S(L))}(f_1\mathds{1}_{L}),f_2\rangle|.
    \end{equation}
    We write $\mathcal{P}_{n-1}\eqqcolon\{L_j^{n-1}\}_{j\in J_{n-1}}$. For $j\in J_{n-1}$, we perform the Whitney decomposition of $\Omega_{L_j^{n-1}}\coloneqq \Omega_1^{L_{j}^{n-1}}\cup \Omega_2^{L_j^{n-1}}$, with
    \[
    \Omega_{i}^{L_j^{n-1}}
    \coloneqq 
    \{x\in\R^d: \mathcal{M}(f_i\mathds{1}_{\Lambda_d L_j^{n-1}})(x)>c_0\langle f_i\rangle_{\Lambda_d L_{j}^{n-1}}\},
    \quad i=1,2.
    \]
obtaining a collection $\{L_{k}^{(n)}\}_{k\in J_{n}(j)}$ for each $j\in J_{n-1}$. Again, $\Omega_{L_j^{n-1}}\subset 2\Lambda_d L_j^{n-1}$ and $\mathcal{M}(f_i\mathds{1}_{\Lambda_d L_j^{n-1}})(x) \lesssim \langle f_i\rangle_{\Lambda_d L_j^{n-1}}$ if $x\not\in \Omega_{L_j^{n-1}}$. Hence, applying Lemma \ref{lemma1.1_Jose}, we obtain
    \begin{align}\label{eq:genericLevelBound}
        \nonumber|\langle T^{(S(L_{j}^{n-1})}(f_1\mathds{1}_{L_j^{n-1}}),f_2\rangle |
        &\leq 
        C\langle f_1\mathds{1}_{L_j^{n-1}}\rangle_{\Lambda_d L_j^{n-1}}\langle f_2\rangle_{\Lambda_d L_j^{n-1}}|\Lambda_d L_j^{n-1}|
        \\&+
        \sum_{k\in J_n(j)}|\langle T^{(S(L_{k}^{(n)}))}(f_1\mathds{1}_{L_j^{n-1}}\mathds{1}_{L_{k}^{(n)}}),f_2\rangle|.
    \end{align}
Reasoning as above, we define 
    \[
    \mathcal{P}_n(j)\coloneqq \{L_{k}^{(n)}: k\in J_{n-1}(j) \text{ for some }j\in J_{n-1}, L_{k}^{(n)}\subset L_j^{n-1}\},
    \]
    and
    \[
    \mathcal{P}_n
    \coloneqq 
    \bigcup_{j\in J_{n-1}} \mathcal{P}_{n}(j)
    \eqqcolon 
    \{L_j^{n}\}_{j\in J_{n}}.
    \]
The cubes in $\mathcal{P}_n$ are pairwise disjoint and
    \begin{equation}\label{eq:LevelNBound}
        |\langle Tf_1,f_2\rangle|
        \leq 
        C\sum_{i=0}^{n-1}\sum_{Q\in\mathcal{P}_i}
        \langle f_1\rangle_{\Lambda_d Q}\langle f_2\rangle_{\Lambda_d Q}|\Lambda_d Q|
        +
        \sum_{L\in \mathcal{P}_{n}}|\langle T^{(S(L))}(f_1\mathds{1}_{L}),f_2\rangle|.
    \end{equation}
Estimate \eqref{eq:LevelNBound} is enough to prove the statement, as we next see. We put $\mathcal{P}\coloneqq \bigcup_{i=0}^{\infty}\mathcal{P}_i$, and we define
\[
\mathcal{S}
\coloneqq 
\{\Lambda_d Q\}_{Q\in \mathcal{P}}
=
\bigcup_{i=0}^{\infty}\mathcal{S}(i)
\]
where $\mathcal{S}(i)\coloneqq \{\Lambda_d Q\}_{Q\in\mathcal{P}_i}$. To check that $\mathcal{S}$ is indeed sparse, we make a choice of disjoint subsets $E_Q$ corresponding to each $Q\in \mathcal{S}$. We first do so for $\Ss(0)\coloneqq\{\Lambda_dQ_0\}$. By the weak $(1,1)$-boundedness of $\mathcal{M}$,
\[
|\Omega_{Q_0}|
\leq 
\frac{2\|\mathcal{M}\|_{L^1\to L^{1,\infty}}}{c_0}|\Lambda_d Q_0|.
\]
Taking 
\[
c_0
=
\frac{2\|\mathcal{M}\|_{L^1\to L^{\infty}}(\Lambda_d)^d}{1-\eta_0},
\]
then 
\[
|\Omega_{Q_0}|
\leq
\frac{1-\eta_0}{2\|\mathcal{M}\|_{L^1\to L^{\infty}}(\Lambda_d)^d}
2\|\mathcal{M}\|_{L^1\to L^{1,\infty}}|\Lambda_d Q_0|
=
(1-\eta_0)|Q_0|,
\]
which allows us to choose the major subset
\[
    E_{\Lambda_dQ_0}
    \coloneqq 
    Q_0\mathbin{\big\backslash}{\bigcup_{j} L_j^{1}}.
\]
It satisfies 
\[
\left|E_{\Lambda_dQ_0}\right| 
\geq
\frac{\eta_0}{(\Lambda_d)^d}|\Lambda_d Q_0|,
\]
because
\[
    \left|
    Q_0\setminus{\bigcup_{j\in J_1} L_j^1}
    \right|
    \geq 
    |Q_0|-\sum_{j\in J_1}|L_j^1|
    =
    |Q_0|-|\Omega_{Q_0}|
    \geq
    |Q_0|-(1-\eta_0)| Q_0|
    =
    \eta_0|Q_0|.    
\]
To the cubes in $\Ss(1)\coloneqq \{\Lambda_dL_j^{1}\}_{j\in J_1}$, we assign the major subsets
    \[
    E_{\Lambda_dL_j^{1}}
    \coloneqq 
    L_j^{1}\setminus{\bigcup_{k\in J_2(j)} L_k^{(2)}}
    =
    L_j^{1}\setminus{\bigcup_{L\in \mathcal{P}_2(j)} L}
    =
    L_j^{1}\setminus{\bigcup_{L\in \mathcal{P}_2} L},
    \quad  j\in J_1, 
    \] 
    which again satisfy $|E_{\Lambda_dL_j^{1}}|\geq \frac{\eta_0}{(\Lambda_d)^d}|\Lambda_d L_j^{1}|$. And in general, to the cubes in $\Ss(n)\coloneqq \{\Lambda_d L_j^{n-1}\}_{j\in J_{n-1}}$ we associate the major subset
    \[
    E_{\Lambda_d L_j^{n-1}}
    \coloneqq 
    L_j^{n-1}\setminus{\bigcup_{k\in J_n(j)}L_k^{(n)}}
    =
    L_j^{n-1}\setminus{\bigcup_{L\in\mathcal{P}_{n}} L}, 
    \]
    and we have $|E_{\Lambda_d L_j^{n-1}}|\geq \frac{\eta_0}{(\Lambda_d)^d}|\Lambda_d L_j^{n-1}|$. Therefore, if we prove that the sets $\{E_{Q}\}_{Q\in\mathcal{S}}$ are disjoint, then $\mathcal{S}$ is $\frac{\eta_0}{(\Lambda_d)^d}$-sparse. To see that $\{E_{\Lambda_d Q}\}_{Q\in\mathcal{P}}$ are pairwise disjoint, let $Q=L_j^n\in\mathcal{P}_n$. It is enough to show that $E_{\Lambda_d Q}$ is pairwise disjoint with $E_{\Lambda_d Q'}$ for every $Q'\in\mathcal{P}_{m}$, $0\leq m\leq n$. First, as we have already observed, the cubes in $\mathcal{P}_n=\{L_j^n\}_{j\in J_n}$ are pairwise disjoint. Hence $E_{\Lambda_d L_j^n}\cap E_{\Lambda_d L_k^n}=\varnothing $ if $j\neq k$, so the major subsets of the cubes of the $n$-th level of $\mathcal{S}$ are disjoint. To pass to the above levels, we observe the following: each cube $L_j^n\in\mathcal{P}_{n}$ is really some $L_{j'}^{(n)}$ with $j'\in J_{n-1}(k)$ for some $k\in J_{n-1}$. That is, 
    \[
    L_{j}^n
    =
    L_{j'}^{(n)}
    \subset
    L_k^{n-1}.
    \]
    Therefore, since $L_j^{n}$ has been removed from $E_{\Lambda_d L_k^{n-1}}$ and $E_{\Lambda_d L_j^{n}}\subset L_j^n$, we have that
    \[
    E_{\Lambda_d L_k^{n-1}}\cap E_{\Lambda_d L_j^{n}}
    =
    \varnothing.
    \]
    But, at the same time, the cube $L_k^{n-1}$ has been removed from the major subsets of the cubes $L_l^{n-2}$, $l\in J_{n-2}$. Hence, the cube $L_j^{n}$, which is contained in $L_k^{n-1}$, has also been removed. Hence, iterating this procedure to reach level $0$, we have that
    \[
    E_{\Lambda_d L_j^n}
    \cap 
    E_{\Lambda_d L_k^{m}}
    =
    \varnothing
    \]  
    for every $k\in J_m$, $0\leq m\leq n-1$. Letting
    \[
    \Ss_n\coloneqq\bigcup_{j=0}^{n}\Ss(n),
    \]
    the estimate \eqref{eq:LevelNBound} at step $n$ can be written as
    \[
    |\langle Tf_1,f_2\rangle|
    \lesssim 
    \langle \mathcal{A}_{\Ss_n}f_1,f_2\rangle 
    +
    \sum_{Q\in \mathcal{P}_n}|\langle T^{(S(Q))}(f_1\one_{Q}),f_2\rangle|.
    \]
    Taking the limit as $n\to\infty$, we obtain the desired estimate. Indeed, the families $\Ss_n$ are increasing, and therefore $\langle \mathcal{A}_{\Ss_n}(f_1),f_2\rangle$ is an increasing sequence that converges to 
    \[
    \langle \mathcal{A}_{\Ss}(f_1),f_2\rangle
    =
    \sum_{Q \in \Ss}\left\langle f_1\right\rangle_Q\left\langle f_2\right\rangle_Q|Q|.
    \]
    The sum in $\mathcal{P}_n$ vanishes when $n\to\infty$, since
    \[
    \left|\sum_{Q\in\mathcal{P}_n}\langle T(f_1\one_Q),f_2\one_{\Lambda_d Q}\rangle \right|
    \lesssim 
    \sum_{Q\in\mathcal{P}_n}\|f_1\one_{Q}\|_2\|f_2\one_{\Lambda_d Q}\|_{2}
    \lesssim_{d}
    \|f_1\|_{\infty}\|f_2\|_{\infty}\sum_{Q\in\mathcal{P}_n}|Q|^{\frac{1}{2}}|Q|^{\frac{1}{2}}
    \xrightarrow{n\to\infty}
    0.
    \]
    To justify that the limit is $0$, notice that
    \[
    \sum_{Q\in\mathcal{P}_n}|Q|
    =
    \sum_{j\in J_{n-1}}
    \left(
    \sum_{L_{k}^{(n)}\in \mathcal{P}_n(j)}
    |L_k^{(n)}|
    \right)
    \leq 
    \sum_{j\in J_{n-1}}
    |\Omega_{L_j^{n-1}}|
    \leq 
    \sum_{j\in J_{n-1}}(1-\eta_0)|L_j^{n-1}|
    =
    (1-\eta_0)\sum_{Q\in\mathcal{P}_{n-1}}|Q|.
    \]
    On the right-hand side we obtain an identical expression to the one on the left-hand side, but now with $n-1$ instead of $n$. Iterating this estimate, we obtain
    \[
    \sum_{Q\in\mathcal{P}_n}|Q|
    \leq 
    (1-\eta_0)^n\sum_{j\in J_1}|L_j^{1}|
    =
    (1-\eta_0)^n|\Omega_{Q_0}|
    \lesssim
    (1-\eta_0)^{n+1}|Q_0|,
    \]
    which tends to $0$ as $n\to\infty$ because $\eta_0\in (0,1)$.  \qedhere    
\end{proof}

\section{Remarks and complements: nonhomogeneous sparse domination} \label{sec:sec3}

We end this note commenting on the differences and challenges that one encounters when transferring the statement of Theorem \ref{theoremA} to other measure spaces, and with some open problems. We assume now that our Calder\'on-Zygmund operator is an operator with integral representation
$$
Tf(x) = \int_{\mathcal{X}} K(x,y) f(y) \, d\mu(y),
$$
where $(\mathcal{X},d,\mu)$ is a metric measure space and the kernel $K$ has a local behavior $|K(x,y)| \sim |x-y|^{-n}$ for some integer $n$. 

\subsection{Homogeneous spaces} $(\mathcal{X},d,\mu)$ is a space of homogeneous type if $\mu$ is doubling with respect to $d$, that is, for each point $x\in\mathcal{X}=\supp{\mu}$ and every $r >0$, there holds
\begin{equation}\label{eq:doub}
\mu(B(x,2r)) \leq C_{\mathrm{doub}} \; \mu(B(x,r)).  
\end{equation}
In that situation, if one formulates the right smoothness condition for $K$ then Theorem \ref{theoremA} holds, as shown for example in \cite{Lo2021}. Moreover, the scheme of proof that we used in Section \ref{sec:sec2} can be used, even without postulating the existence of a nice replacement of the dyadic system, as shown in \cite{CADPPV2022}.

\subsection{Calder\'on-Zygmund operators with nondoubling measures} When \eqref{eq:doub} fails, the situation changes dramatically and the challenges that appear if one tries to formulate Theorem \ref{theoremA} in this situation are significant. Let us focus on the case $\mathcal{X}=\R^d$. There is a rich Calder\'on-Zygmund theory which was developed for measures of polynomial growth. What is postulated in that case is that there exists $0<n\leq d$ so that
$$
\mu(B(x,r)) \leq r^n
$$
holds for all $x\in\supp{\mu}$ and all $r>0$. The size growth of the singularity of Calder\'on-Zygmund kernels matches that of the measure, and a fairly complete transference of the classical results in Calder\'on-Zygmund theory is available, including weak-type estimates \cite{NTV1997,To2001b}, BMO spaces \cite{To2001}, and $Tb$ theorems \cite{NTV2003}. A more general theory is even possible, with measures satisfying milder assumptions \cite{Hy2010}. There are two results that transfer sparse domination to this setting, each following one of the methods laid out in the Introduction:
\begin{itemize}
    \item \textbf{Median oscillation formula:} The median oscillation formula was extended to general measures by H\"anninen \cite{Ha2017}. Using that and a system of dyadic cubes adapted to $\mu$ \cite{DM2000} one can follow the original scheme of \cite{Le2013}, obtaining a pointwise bound
    $$
    |Tf(x)|\lesssim \AS(\tilde{\mathcal{M}}f)(x),
    $$
    for some adequate modified maximal function $\tilde{\mathcal{M}}$ \cite{CoPa2019}. The drawback of this approach is that, even when $\mu$ is doubling, one does not recover the $A_2$ theorem, for the quantitative weighted dependence is suboptimal (one gets a quadratic dependence on the $A_2$ constant of the weight).
    \item \textbf{Grand maximal function:} Using once more the David-Mattila dyadic system from \cite{DM2000}, \cite{VoZK2018} follows the proof of \cite{Le2016} to obtain a pointwise estimate of the form
    $$
    |Tf(x)|\lesssim \sum_j a_j \tilde{\mathcal{A}}_{\Ss_j}|f|(x),
    $$
    where $\{a_j\}_j$ is summable and the modified sparse operators are of the form 
    $$
    \tilde{\mathcal{A}}_{\Ss} f(x) =\sum_{Q \in \Ss} \langle f\rangle_{\alpha Q} \one_Q(x),
    $$
    for some dilation parameter $\alpha$. The result does recover the $A_2$ theorem in the doubling case, but is not comparable with \cite{CoPa2019} in the general case, which shows that it is also suboptimal.
\end{itemize}

It is natural to wonder whether the arguments from Section \ref{sec:sec2} can be carried over to this setting. There exists a dyadic Calder\'on-Zygmund decomposition that works for general measures \cite{CCP2022}, so an estimate like the one in the iteration Lemma \ref{lemma1.1_Jose} cannot be ruled out. However, it is unclear whether the geometric technique of the Whitney covering, which is critically used to run the main stopping-time argument, admits an adequate replacement.

\begin{problem}
Find an iteration scheme to prove a version of Theorem \ref{theoremA} for measures of polynomial growth using the Calder\'on-Zygmund decomposition.
\end{problem}

One can avoid the sparse domination technique altogether and look for a path to generalize the $A_2$ theorem using different methods. Hyt\"onen's representation theorem \cite{Hy2012} states that a Calder\'on-Zygmund operator can be written as an average of certain dyadic operators, called dyadic shifts, which are generally considered to be discrete models of singular integrals. Haar shifts defined with respect to general measures were considered in \cite{LSMP2012}, where a characterization of the class of measures $\mu$ for which they are of weak-type $(1,1)$ was found. The corresponding class of measures fails to include all polynomial growth measures, which is an indication that a direct generalization of Hyt\"onen's representation is unlikely (see also \cite{deLaHerran}, where a representation theorem is proved that allows one to prove a $T1$ theorem, but not sharp weighted inequalities). 

\begin{problem}
Find the right generalization of Hyt\"onen's representation theorem for Calder\'on-Zygmund operators defined with respect to measures of polynomial growth, and prove a nonhomogeneous $A_2$ theorem.
\end{problem}

Even for Haar shifts that are of weak-type $(1,1)$ sparse domination fails, and more complicated averaging operators need to be considered in the right hand side of the sparse domination inequality in order to obtain a positive result \cite{CPW2023}. This yields interesting weighted inequalities as a corollary and it suggests that one can try a similar approach for Calder\'on-Zygmund operators.

\begin{problem}
Define an averaging operator $\mathcal{B}_\Ss$ so that the result
$$
|Tf(x)| \lesssim \AS|f|(x) + \mathcal{B}_\Ss|f|(x)
$$
holds and entails quantitatively sharp weighted inequalities. 
\end{problem}

\begin{remark}
When looking for weighted inequalities in the nonhomogeneous setting, it may be relevant to notice that the class of $A_p$ weights --defined with respect to the measure $\mu$-- fails to characterize $L^p$-boundedness of Calder\'on-Zygmund operators, as shown in \cite{To2007}. This may be an indication that the $A_p$ class and hence sparse operators --whose definition is adapted to $A_p$ weights-- may not be the right tools to study weighted bounds when the doubling condition fails.   
\end{remark} 

\noindent\textbf{Conflict of interest:} The authors report no conflict of interest. 

\bibliographystyle{alpha}
\bibliography{references}

\end{document}